\documentclass{article}

\usepackage{arxiv}

\usepackage[utf8]{inputenc} 
\usepackage[T1]{fontenc}    
\usepackage{hyperref}       
\usepackage{url}            
\usepackage{booktabs}       
\usepackage{amsfonts}       
\usepackage{nicefrac}       

\usepackage{microtype}      
\usepackage{lipsum}		
\usepackage{graphicx}
\usepackage{natbib}
\usepackage{doi}

\usepackage{amsthm}
\usepackage{amsmath}
\usepackage{float}
\usepackage{subcaption}
\newtheorem{theorem}{Theorem}
\newtheorem{lemma}[theorem]{Lemma}

\title{A note on the random constrained-degree percolation model on $\mathbb{L}^2$}

\author{ \href{https://orcid.org/0000-0003-1209-9944}{\includegraphics[scale=0.06]{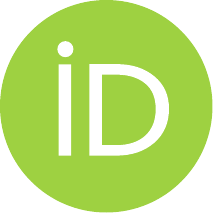}\hspace{1mm} Marco A. Ticse A.}\\
	Departamento de Matemática\\
	Universidade Federal de Minas Gerais\\
	\texttt{markito2016@ufmg.br} \\
}

\renewcommand{\undertitle}{}
\renewcommand{\shorttitle}{A note on the random constrained-degree percolation model on $\mathbb{L}^2$}

\begin{document}
\maketitle

\begin{abstract}
This note considers a constrained-degree percolation model in a random environment on the square lattice, where each vertex is assigned a random degree constraint, and edges attempt to open at random times. We prove that whenever an infinite open cluster exists, it is almost surely unique, implying the positivity of the connectivity function in the supercritical regime.
\end{abstract}


\section{Introduction}

Percolation models on regular lattices have been widely used to describe propagation, connectivity, and diffusion phenomena in random media (\cite{MR3389781,MR4200847}). Among them, the classical Bernoulli percolation model stands out, in which edges are declared open independently with probability \(p\). However, in several situations of interest, the independence assumption between the states of the connections is not satisfied, which has motivated the development and study of dependent percolation models (\cite{Gaunt_1979,MR2676027,MR4108128,MR4321220}).

In this note, we study the Constrained-Degree Percolation model in Random Environment (CDPRE), introduced in \cite{MR4492964}. The model assigns an independent random degree constraint to each vertex, while edges attempt to open at random times provided that the degree constraints at their endpoints are respected. It was shown in \cite{MR4492964} that the CDPRE model exhibits a non-trivial supercritical phase for sufficiently permissive degree distributions and that correlations between local events decay exponentially with distance.

One of the central problems in percolation theory is establishing the uniqueness of the infinite cluster in the supercritical phase, which makes it possible to establish various properties, such as the positivity of the connectivity function and the differentiability of the percolation probability., among others (\cite{MR1707339}). For the deterministic constrained-degree percolation model, where every vertex has degree constraint three, this property was first established on the square lattice in \cite{MR4127336} and later extended to higher dimensions in \cite{arcanjo2026constrained}. The latter work also established positivity of the connectivity function and differentiability of the percolation probability.

Along those lines, we extend the uniqueness result to the random-environment setting on the square lattice. We prove that, whenever an infinite open cluster exists, it is almost surely unique. As a consequence, we establish the positivity of the connectivity function in the supercritical regime.

\section{Model}

We introduce the CDPRE model on $\mathbb{L}^2=(\mathbb{Z}^2,\mathcal{E}^2)$. Consider a collection of random variables $\kappa=\{\kappa_v\}_{v\in\mathbb{Z}^2}$, where each $\kappa_v$ represents the degree constraint associated with the vertex $v$. The variables $\kappa_v$ are assumed to be \textit{i.i.d.} and take values $j\in\{0,1,2,3\}$ according to product measure $\mathbb{P}_{\rho}(\kappa_v=j)=\rho_j$ on $\{0,1,2,3\}^{\mathbb{Z}^2}$.

We also consider a collection $U=\{U_e\}_{e\in\mathcal{E}^2}$, where the variables $U_e$ are \textit{i.i.d.} and uniformly distributed on $[0,1]$, independently of $\kappa$. Let $\mathbb{P}$ denote the corresponding product measure on $[0,1]^{\mathcal{E}^2}$. The percolation configuration at time $t\in[0,1]$ is given by
\begin{align*}
\omega_t \colon \{0,1,2,3\}^{\mathbb{Z}^2} \times [0, 1]^{\mathcal{E}^2} &\longrightarrow \ \{0, 1\}^{\mathcal{E}^2} \\
(\kappa, U) \ \ \ \  \quad  &\longmapsto \omega_t(\kappa, U)
\end{align*}
where $\omega_t(\kappa,U)$ denotes the edge configuration at time $t$. The dynamics of the process are defined as follows:
\begin{itemize}
    \item At time $t=0$, all edges are closed;
    \item Each edge $e=\langle u,v\rangle$ attempts to open at time $U_e$. The attempt is successful if, at that time, both vertices $u$ and $v$ have degrees strictly smaller than their corresponding degree constraints. That is, $\deg(u,U_e)<\kappa_u$ and $\deg(v,U_e)<\kappa_v$, where $\deg(v,t)$ denotes the number of open edges incident to $v$ at time $t$;
    \item Once an edge becomes open, it remains open forever.
\end{itemize}
Furthermore, we denote by $\omega_{t,e}$ the state of edge $e$ at time $t$, where $\omega_{t,e}=1$ if $e$ is open and $\omega_{t,e}=0$ otherwise. 
\begin{figure}
\centering
\begin{subfigure}[b]{0.2865\linewidth}
\includegraphics[width=\linewidth]{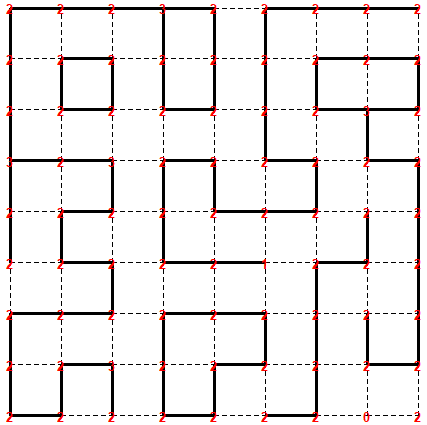}
\caption{ $\rho = (0.025, 0.025, 0.9, 0.05, 0.0)$}
\label{fig:westminster_lateral}
\end{subfigure}
\hspace{2cm}
\begin{subfigure}[b]{0.28\linewidth}
\includegraphics[width=\linewidth]{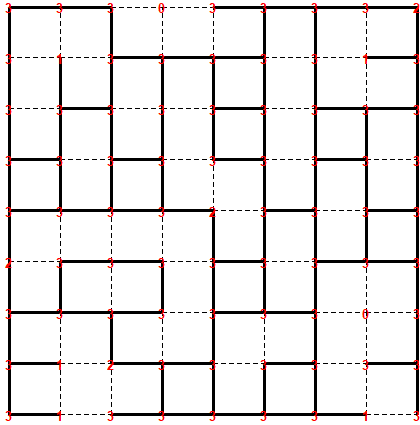}
\caption{$\rho = (0.025, 0.025, 0.05, 0.9, 0.0)$}
\label{fig:westminster_aerea}
\end{subfigure}
\caption{A sample configuration of the CDPRE model with parameter $\rho$ at time $t=1$.}
\label{figuraro}
\end{figure}

Finally, the percolation process described above is governed by the probability measure
\[
\mathbb{P}_{\rho,t}(A)
=
(\mathbb{P}_{\rho}\times\mathbb{P})
\bigl(
\{(\kappa,U):\omega_t(\kappa,U)\in A\}
\bigr),
\]
for every event $A\subset\{0,1\}^{\mathcal{E}^2}$. We denote by $\theta(\rho,t)$ the probability that the origin belongs to an infinite open cluster at time $t$, and by $t_c(\rho)$ the critical point of the model, namely,
\[
\theta(\rho,t)
=
\mathbb{P}_{\rho,t}(0\longleftrightarrow\infty)
\qquad \text{and} \qquad
t_c(\rho)
=
\sup\{t\in[0,1]:\theta(\rho,t)=0\}.
\]
The CDPRE model differs from classical percolation in that, although it is translation invariant, it exhibits long-range dependencies among edge states and does not satisfy the FKG inequality (see \cite{YANG2000203,Glagov2024}).

\section{Main Results}

\begin{theorem}[Uniqueness]\label{unique}
Consider the CDPRE model on $\mathbb{L}^2$ with $\rho=(\rho_0,\rho_1,\rho_2,\rho_3)$ and $\rho_3>0$. If $\theta(\rho,t)>0$, then  the infinite open cluster is almost surely unique.
\end{theorem}

The proof of Theorem~\ref{unique} is inspired by the uniqueness proof for the case in which all degree constraints are equal to three (see \cite{MR4127336}). Unlike the deterministic setting, random degree constraints introduce additional difficulties. To address them, we establish Lemma~\ref{lema.unicidade}, which controls the occurrence of certain local configurations. Especially, the proof is based on assuming the existence of multiple infinite clusters and constructing a local modification inside a sufficiently large box that creates a trifurcation with positive probability. Translation invariance implies that trifurcations occur with positive density, while the \cite{MR990777} argument bounds their number by the size of the boundary of the box. This contradiction yields the uniqueness of the infinite cluster.

A natural consequence of the uniqueness of the infinite cluster is the study of the behavior of the \emph{connectivity function}, defined as the probability that two vertices belong to the same open cluster in a percolation process. In the context of the CDPRE model, the connectivity function between two vertices $u$ and $v$ is given by
\begin{equation*}
    \tau_{\rho,t}(u,v)
    =
    \mathbb{P}_{\rho,t}(u \longleftrightarrow v).
\end{equation*}
Our second result establishes that the connectivity function remains uniformly positive in the supercritical phase.

\begin{theorem}[Connectivity Function]\label{funconec}
Consider the CDPRE model on $\mathbb{L}^2$ with $\rho=(\rho_0,\rho_1,\rho_2,\rho_3)$ and suppose that $\theta(\rho,t)>0$. Then there exists a constant $c(\rho,t)>0$ such that
\begin{equation*}
    \tau_{\rho,t}(u,v)\geq c(\rho,t),
\end{equation*}
for all $u,v\in\mathbb{Z}^2$.
\end{theorem}

The proof combines the Theorem \ref{unique} with the exponential decay of correlations (see \cite{MR4492964}). Positivity for nearby vertices follows from a local argument. For distant vertices, we approximate the event of belonging to the infinite cluster by connection to the boundary of a sufficiently large box and exploit the decay of correlations to obtain an approximate independence estimate. This yields a uniform positive lower bound for the connectivity function.

\section{Preliminary results}

The following lemma, fundamental to our analysis, relies on the notion of the \emph{number of ends} of a graph $\mathcal{G}$, which informally counts its distinct infinite directions. We first define the event
\[
\mathcal{D}_m
=
\{\text{there exist exactly $m$ infinite open clusters}\}.
\]

\begin{lemma}\label{lema.unicidade}
If $\mathbb{P}_{\rho,t}(\mathcal{D}_m)=1$ for some
$m=3,4,\dots,\infty$, then, for all sufficiently large $n$, there exists, with positive probability, an open cluster in the box $B_n$ having at least three ends.
\end{lemma}

\begin{proof} Consider the event 
\begin{equation}\label{evento}
\mathcal{A}_n^t= \{\text{at least three infinite open clusters in }\omega_t\text{ intersect } B_n\}. 
\end{equation} 
For $\ell\in\{n-1,n,n+1\}$ and $e=\langle x,y\rangle\in\mathcal{E}^2$, define 
\begin{align*} \mathcal{O}_{n,\ell} &= \{e\in\mathcal{E}^2 : x\in\partial B_n \text{ and } y\in\partial B_\ell\},\\[0.5em] I_n &= \{e\in\mathcal{O}_{n,n} : e\cap\{(n,n),(-n,n),(-n,-n),(n,-n)\}\neq\emptyset\},\\[0.5em] \mathcal{O}_n &= \{e\in\mathcal{O}_{n,n+1} : e\cap\{(n,n),(-n,n),(-n,-n),(n,-n)\}\neq\emptyset\}. 
\end{align*} 
We also consider the event $ \mathcal{B}_n^t= \{\text{all edges in }I_n\text{ are open at time }t\}.$ It is important to note that there exists a positive constant $c=c(\rho,t)$ such that $\inf_{n\geq 1}\mathbb{P}_{\rho,t}(\mathcal{B}_n^t)>c(\rho,t).$  
Moreover, since $\lim_{n\to\infty}\mathbb{P}_{\rho,t}(\mathcal{A}_n^t)=1,$ it follows that, for all sufficiently large $n$, \[ \mathbb{P}_{\rho,t}(\mathcal{A}_n^t\cap\mathcal{B}_n^t)\geq c(\rho,t). \] For each subset $A\subset\mathcal{O}_{n,n+1}$, define 
$\mathcal{C}_{n,A}^t= \left\{ \omega_t(e)=1 \text{ for } e\in A,\; \omega_t(e)=0 \text{ for } e\in \mathcal{O}_{n,n+1}\setminus A \right\}.$
In other words, $\mathcal{C}_{n,A}^t$ represents the configuration in which the edges in $A$ are open at time $t$, while all remaining edges in $\mathcal{O}_{n,n+1}$ are closed. Since the family $\{\mathcal{C}_{n,A}^t\}_{A\subset\mathcal{O}_{n,n+1}}$ forms a partition of the configuration space, there necessarily exists a subset $A\subset\mathcal{O}_{n,n+1}$ such that \begin{equation}\label{PABC} 
\mathbb{P}_{\rho,t} \left( \mathcal{A}_n^t \cap \mathcal{B}_n^t \cap \mathcal{C}_{n,A}^t \right)>0. \end{equation} 
For convenience, we write 
\[ \mathcal{H}_{n,A}^t := \mathcal{A}_n^t \cap \mathcal{B}_n^t \cap \mathcal{C}_{n,A}^t. \]
Let $\delta\in(0,1)$, and let $\mathcal{U}_{n,\delta}$ denote the set of configurations $U$ such that $ U_e\in(\delta,1-\delta]$ for all $e\in\mathcal{O}_{n,n+1}.$ By \eqref{PABC} and the fact that $\lim_{\delta\to0^+} \mathbb{P}_{\rho,t}(\mathcal{U}_{n,\delta})=1,$ we can choose $\delta<t$ sufficiently small so that 
\begin{equation}\label{PHU} 
\mathbb{P}_{\rho,t} \left( \mathcal{H}_{n,A}^t \cap \mathcal{U}_{n,\delta} \right)>0. 
\end{equation}
Let $u$ be a unit vector with respect to the $\ell_1$-norm. For an edge
$e=\langle x,x+u\rangle\in\mathcal{O}_{n,n+1}\setminus\mathcal{O}_n$,
we denote by $e^o=\langle x,x-u\rangle$ the \emph{edge opposite to $e$}, that is, the edge pointing in the opposite direction. Moreover, if
$M\subset\mathcal{O}_{n,n+1}\setminus\mathcal{O}_n$, we define $M^o=\{e^o:e\in M\}$ as the set of edges opposite to those in $M$.

Finally, we define $A^f=(\mathcal{O}_{n,n+1}\setminus A)^o\cup\mathcal{O}_{n,n}$. The set $A^f$ combines the edges in $\mathcal{O}_{n,n}$ with the edges opposite to those not belonging to $A$. Let
\begin{equation*}
\begin{aligned}
\mathcal{F}_{n,A,\delta}
=
\Big\{
(\kappa,U)\in \{0,1,2,3\}^{B_n} \times [0,1]^{B_n} :\;&
\kappa_v=3
\text{ for all }
v\in \partial B_n\cup\partial B_{n-1},\\
&
U_e\leq\delta
\text{ for all }
e\in A^f, \\
&
U_e\geq 1-\delta
\text{ for all }
e\in B_n\setminus A^f
\Big\}.
\end{aligned}
\end{equation*}
Thus, all degree constraints on the boundaries of $B_n$ and $B_{n-1}$ are fixed at $3$, while the random times $U_e$ satisfy $U_e\leq\delta$ for all $e\in A^f$ and $U_e\geq 1-\delta$ for all $e\in B_n\setminus A^f$. This fixes the edge configuration. 

Given $V\subset\mathbb{Z}^2$ and $F\subset\mathcal{E}^2$, let
$\Pi_{V\times F}$ denote the canonical projection from
$\{0,1,2,3\}^{\mathbb{Z}^2}\times[0,1]^{\mathcal{E}^2}$ onto
$\{0,1,2,3\}^V\times[0,1]^F$. Writing $G_n=B_n^c\times B_n^c$, 
and noting that $\rho_3>0$, equation~\eqref{PHU} implies that
\begin{equation}\label{positividade}
\begin{aligned}
\mathbb{P}_{\rho,t}
\left(
\mathcal{F}_{n,A,\delta}
\times
\Pi_{G_n}
\bigl(
\mathcal{H}_{n,A}^t
\cap
\mathcal{U}_{n,\delta}
\bigr)
\right) =
\mathbb{P}_{\rho,t}
\left(
\mathcal{F}_{n,A,\delta}
\right)
\mathbb{P}_{\rho,t}
\left(
\Pi_{G_n}
\bigl(
\mathcal{H}_{n,A}^t
\cap
\mathcal{U}_{n,\delta}
\bigr)
\right)
>0.
\end{aligned}
\end{equation}
Therefore, if $\omega_t\in \mathcal{F}_{n,A,\delta}\times
\Pi_{G_n}\bigl(\mathcal{H}_{n,A}^t\cap\mathcal{U}_{n,\delta}
\bigr),$ then all edges in $\mathcal{O}_{n,n}$ are open in $\omega_t$, and at least three infinite open clusters intersect $B_n$. This completes the proof. \hfill $\square$
\end{proof}


\section{Proofs}

\subsection{Proof of Theorem \ref{unique}}

Assume that $\mathbb{P}_{\rho,t}(\mathcal{D}_m)=1$ for some
$m=3,4,\dots,\infty$. Then, by Lemma~\eqref{lema.unicidade}, the assumptions of Lemma~7.7 in \cite{MR3616205} are satisfied. Consequently, there exists (on an enlarged probability space) a random forest $\mathfrak{F}\subset\omega_t$ such that:
\begin{itemize}
    \item[(a)] the distribution of the pair $(\omega_t,\mathfrak{F})$ is translation invariant;
    \item[(b)] the collection of forests $\mathfrak{F}$ containing a connected component with at least three ends has positive probability.
\end{itemize}
Let $\mathbb{Q}_t$ denote the distribution of the pair $(\omega_t,\mathfrak{F})$, and let $\mathbb{E}_t[\cdot]$ be the corresponding expectation. Let $X$ be the set of trifurcation points of $\mathfrak{F}$; that is, $x\in X$ if removing all edges incident to $x$ splits the connected component containing $x$ into at least three infinite connected components. By item (b), we have $\mathbb{Q}_t(v\in X)>0$ for some $v\in\mathbb{Z}^2$. Translation invariance then implies that
\[
\mathbb{E}_t\bigl[|X\cap B_n|\bigr]
=
|B_n|\,\mathbb{Q}_t(0\in X).
\]
On the other hand, the Burton--Keane argument yields $|B_n\cap X| \leq|\partial B_n|$ (see \cite{MR990777}). Combining the two previous expressions, we obtain the existence of a constant $c>0$ such that
\[
|\partial B_n|
\geq
\mathbb{E}_t\bigl[|B_n\cap X|\bigr]
=
c|B_n|,
\]
for all sufficiently large $n$, which is a contradiction.

It remains to show that $\mathbb{P}_{\rho,t}(\mathcal{D}_2)=0$. This follows directly from \eqref{positividade} by replacing the definition of $\mathcal{A}_n^t$ in \eqref{evento} with
\[
\mathcal{A}_n^t=
\{\text{exactly two infinite open clusters in }\omega_t\text{ intersect }B_n\}.
\]
Indeed, the occurrence of $\mathcal{F}_{n,A,\delta} \times
\Pi_{G_n}\bigl(\mathcal{H}_{n,A}^t \cap \mathcal{U}_{n,\delta}
\bigr)$ forces all edges in $\partial B_n$ to be open, thereby merging the two infinite clusters. Hence,
\[
\mathbb{P}_{\rho,t} \left(\left(\bigcup_{m=2}^{\infty}\mathcal{D}_m\right)\cup
\mathcal{D}_\infty\right)=0.
\]
Finally, since the measure $\mathbb{P}_{\rho,t}$ is translation invariant and ergodic, and   $\mathcal{D}_m$ is translation invariant, the number of infinite open clusters is almost surely constant. Therefore, the infinite open cluster is almost surely unique. \hfill $\square$

\subsection{Proof of Theorem  \ref{funconec}}

Since the measure $\mathbb{P}_{\rho,t}$ is translation invariant, it follows that, for any $u,v\in\mathbb{Z}^2$,
\begin{equation}\label{uv=z}
\tau_{\rho,t}(u,v)
:=\mathbb{P}_{\rho,t}(u\longleftrightarrow v)
=\mathbb{P}_{\rho,t}(0\longleftrightarrow v-u)
=\tau_{\rho,t}(0,z),   
\end{equation}
where $z=v-u$. Therefore, it suffices to study the function $\tau_{\rho,t}(0,z)$. By Theorem~\ref{unique}, whenever $\theta(\rho,t)>0$, the infinite open cluster is almost surely unique. Hence,
\begin{equation}\label{conce.uniq1}
\mathbb{P}_{\rho,t}(0\longleftrightarrow z)
\geq
\mathbb{P}_{\rho,t}(0\longleftrightarrow\infty,\,
z\longleftrightarrow\infty).
\end{equation}
 Denote by $B_n(x)$ the box of side length $2n$ centered at $x\in\mathbb{Z}^2$. For simplicity, we write $B_n:=B_n(0)$.  Let $C$ denote the open cluster containing the origin. Observe that 
\[
\{0\longleftrightarrow\infty,\,
z\longleftrightarrow\infty\}
\subset
\{0\longleftrightarrow\partial B_n,\,
z\longleftrightarrow\partial B_n(z)\},
\]
and that the difference between these two events is contained in
\begin{equation}\label{diferença}
    \{0\longleftrightarrow\partial B_n,\,
z\longleftrightarrow\partial B_n(z),\,
|C|<\infty\}
\cup
\{0\longleftrightarrow\partial B_n,\,
z\longleftrightarrow\partial B_n(z),\,
|C_z|<\infty\}.    
\end{equation}
Using \eqref{diferença} together with translation invariance, we obtain
\begin{equation*}
\begin{aligned}
\mathbb{P}_{\rho,t}( 0 \longleftrightarrow \partial B_n , z \longleftrightarrow \partial &B_n(z))
-\mathbb{P}_{\rho,t}(0 \longleftrightarrow \infty, z \longleftrightarrow \infty) \\
&\leq \mathbb{P}_{\rho,t}(0 \longleftrightarrow \partial B_n, z \longleftrightarrow \partial B_n(z), |C|<\infty) + \mathbb{P}_{\rho,t}(0 \longleftrightarrow \partial B_n, z \longleftrightarrow \partial B_n(z), |C_z|<\infty) \\
&= 2 \  \mathbb{P}_{\rho,t}(0 \longleftrightarrow \partial B_n, z \longleftrightarrow \partial B_n(z), |C|<\infty) \\
&\leq 2 \  \mathbb{P}_{\rho,t}(0 \longleftrightarrow \partial B_n, |C|<\infty).
\end{aligned}
\end{equation*}
Since the events $\{0\longleftrightarrow\partial B_n,\,
|C|<\infty\}$ decrease to the empty set as $n\to\infty$, it follows that
\[
\mathbb{P}_{\rho,t}\bigl(0\longleftrightarrow\infty,\,
z\longleftrightarrow\infty\bigr)
\geq\mathbb{P}_{\rho,t}\bigl(
0\longleftrightarrow\partial B_n,\,z\longleftrightarrow\partial B_n(z)\bigr)-
2\,\mathbb{P}_{\rho,t}\bigl(0\longleftrightarrow\partial B_n,\,
|C|<\infty\bigr).
\]
Hence, there exists $R$ sufficiently large, independent of $z$, such that
\[
\mathbb{P}_{\rho,t}\bigl(0\longleftrightarrow\partial B_R,\,
|C|<\infty\bigr)<\frac{\theta(\rho,t)^2}{4}.
\]
Therefore,
\begin{equation}\label{conce.uniq2}
\mathbb{P}_{\rho,t}
\bigl(
0\longleftrightarrow\infty,\,
z\longleftrightarrow\infty
\bigr)
>
\mathbb{P}_{\rho,t}
\bigl(
0\longleftrightarrow\partial B_R,\,
z\longleftrightarrow\partial B_R(z)
\bigr)
-
\frac{\theta(\rho,t)^2}{2}.
\end{equation}
By \eqref{conce.uniq1} and \eqref{conce.uniq2}, we obtain
\[
\tau_{\rho,t}(0,z)
>
\mathbb{P}_{\rho,t}
\bigl(
0\longleftrightarrow\partial B_R,\,
z\longleftrightarrow\partial B_R(z)
\bigr)
-
\frac{\theta(\rho,t)^2}{2}.
\]

For notational convenience, let $\mathcal{A}_{x,n} =\{x\longleftrightarrow\partial B_n(x)\}.$
By Theorem 2 (in \cite{MR4492964}), there exist constants $\psi_1(d)>0$ and $\psi_2(d)>0$, depending only on the dimension $d$, such that
\[
\left|
\mathbb{P}_{\rho,t}
\bigl(
\mathcal{A}_{0,R}\cap\mathcal{A}_{z,R}
\bigr)
-
\mathbb{P}_{\rho,t}(\mathcal{A}_{0,R})
\mathbb{P}_{\rho,t}(\mathcal{A}_{z,R})
\right|
\leq
\psi_1(d)\,R\,
e^{-\psi_2(d)\delta(B_R,B_R(z))}.
\]
Using translation invariance, for $d=2$ we deduce that
\begin{align*}
\mathbb{P}_{\rho,t}
\bigl(\mathcal{A}_{0,R}\cap\mathcal{A}_{z,R}\bigr)
&\geq
\big(\mathbb{P}_{\rho,t}(\mathcal{A}_{0,R})\big)^2-
\psi_1(2)\,R\,e^{-\psi_2(2)\delta(B_R,B_R(z))}
\geq
\theta(\rho,t)^2 - \psi_1(2)\,R\,e^{-\psi_2(2)\delta(B_R,B_R(z))}.
\end{align*}
Consequently, if $0$ and $z$ are sufficiently far apart so that
\[
e^{-\psi_2(2)\delta(B_R,B_R(z))}
<
\frac{\theta(\rho,t)^2}
{4\psi_1(2)R},
\]
then \eqref{conce.uniq2} implies that
\begin{equation}\label{c1}
\tau_{\rho,t}(0,z) > \frac{\theta(\rho,t)^2}{4}.    
\end{equation}

To complete the proof, it remains to consider the case where $z$ is close to the origin. Fix $R>0$ as above and define $\Gamma_{\rho,t}(R) = \left\{ \tau_{\rho,t}(0,z) : z\in B_R \right\}$. Since $B_R$ is finite, the set $\Gamma_{\rho,t}(R)$ is also finite. 
Moreover, for every $z\in B_R$, $\tau_{\rho,t}(0,z)>0$ since any finite path connecting $0$ and $z$ is open with positive probability. Thus, the quantity 
\begin{equation}\label{c2}
c_1(\rho,t,R) := \min_{z\in B_R} \tau_{\rho,t}(0,z)
\end{equation}
is well defined and strictly positive. Hence, by \eqref{c1} and \eqref{c2}, defining
\[
c(\rho,t):=\min\left\{\frac{\theta(\rho,t)^2}{4},\,c_1(\rho,t,R)\right\},
\]
we obtain $\tau_{\rho,t}(0,z)\geq c(\rho,t)$ for all $z\in\mathbb{Z}^2$. The result then follows from translation invariance and \eqref{uv=z}.
\hfill $\square$


\bibliographystyle{cas-model2-names}

\bibliography{references}  






\end{document}